\documentclass[10pt,reqno,a4paper,final]{amsart}
\usepackage[british]{babel}
\usepackage{amssymb,amstext,amsthm,eucal,bbm,amscd,bbold}
\usepackage[margin=1in]{geometry}
\usepackage{array}
\usepackage[svgnames,table]{xcolor}
\usepackage[T1]{fontenc}
\usepackage[utf8x]{inputenc}
\usepackage[small]{eulervm}
\usepackage{mathrsfs}
\usepackage{verbatim}
\usepackage{pdfsync}
\usepackage[final]{microtype}
\usepackage[nosort]{cite}
\usepackage{booktabs} 
\usepackage{caption}
\usepackage{appendix}
\usepackage{chngcntr}
\usepackage{apptools}
\usepackage{tikz,pgfplots}
\usepgfplotslibrary{fillbetween}
\usetikzlibrary{calc,arrows,arrows.meta,cd,shapes.geometric,positioning,through,decorations.markings,decorations.text}
\tikzset{>=latex}
\usepackage[unicode,pagebackref=true]{hyperref}
\hypersetup{%
  pdftitle   = {Lie algebraic Carroll/Galilei duality},
  pdfkeywords = {Galilei, Carroll, Bargmann, double extension, Lie algebras, Lie groups},
  pdfauthor  = {José Figueroa-O'Farrill},
  pdfcreator = {\LaTeX\ with package \flqq hyperref\frqq},
  linkcolor=NavyBlue,
  citecolor=ForestGreen,
  urlcolor=OrangeRed,
  anchorcolor=OrangeRed,
  colorlinks=true
}
\renewcommand*{\backref}[1]{}
\renewcommand*{\backrefalt}[4]{%
  \ifcase #1%
  \or [Page~#2.]%
  \else [Pages~#2.]%
  \fi%
}
\theoremstyle{plain}
\newtheorem{lemma}{Lemma}
\newtheorem{proposition}[lemma]{Proposition}
\newtheorem{theorem}[lemma]{Theorem}

\theoremstyle{definition}
\newtheorem{definition}[lemma]{Definition}

\newcommand{\g}{\mathfrak{g}}
\newcommand{\ghat}{\widehat{\mathfrak{g}}}
\newcommand{\gbar}{\overline{\mathfrak{g}}}
\newcommand{\h}{\mathfrak{h}}
\renewcommand{\a}{\mathfrak{a}}
\renewcommand{\b}{\mathfrak{b}}
\renewcommand{\c}{\mathfrak{c}}

\newcommand{\ses}{\mathfrak{ss}}
\newcommand{\heis}{\mathfrak{heis}}
\newcommand{\nw}{\mathfrak{nw}}
\newcommand{\so}{\mathfrak{so}}
\newcommand{\iso}{\mathfrak{iso}}

\newcommand{\s}{\mathfrak{s}}

\newcommand{\be}{\boldsymbol{e}}

\newcommand{\eX}{\mathcal{X}}

\newcommand{\ad}{\operatorname{ad}}
\newcommand{\Ort}{\operatorname{O}}

\newcommand{\RR}{\mathbb{R}}
\newcommand{\MM}{\mathbb{M}}

\newcommand{\GL}{\operatorname{GL}}
\newcommand{\SO}{\operatorname{SO}}

\newcommand{\psibar}{\overline{\psi}}
\newcommand{\Dbar}{\overline{D}}
\newcommand{\Ubar}{\overline{U}}
\newcommand{\Xbar}{\overline{X}}
\newcommand{\Ybar}{\overline{Y}}

\newcommand{\zLC}{\mathsf{LC}}
\newcommand{\zAdSC}{\mathsf{AdSC}}
\newcommand{\zdSC}{\mathsf{dSC}}

\newcommand{\zC}{\mathsf{C}}
\newcommand{\zG}{\mathsf{G}}

\newcommand{\zAdSG}{\mathsf{AdSG}}
\newcommand{\zdSG}{\mathsf{dSG}}

\definecolor{dkgr}{rgb}{0,0.6,0}
\definecolor{gris}{rgb}{0.5,0.5,0.5}
\newcommand{\zero}{{\color{gris}0}}
\allowdisplaybreaks[0]
\numberwithin{equation}{section}
\AtAppendix{\counterwithin{lemma}{section}}

\begin{document}

\title{Lie algebraic Carroll/Galilei duality}

\author{José Figueroa-O'Farrill}
\address{Maxwell Institute and School of Mathematics, The University
  of Edinburgh, James Clerk Maxwell Building, Peter Guthrie Tait Road,
  Edinburgh EH9 3FD, Scotland, United Kingdom}
\email{\href{mailto:j.m.figueroa@ed.ac.uk}{j.m.figueroa@ed.ac.uk}}
\begin{abstract}
  We characterise Lie groups with bi-invariant bargmannian, galilean
  or carrollian structures.  Localising at the identity, we show that
  Lie algebras with ad-invariant bargmannian, carrollian or galilean
  structures are actually determined by the same data: a metric Lie
  algebra with a skew-symmetric derivation.  This is the same data
  defining a one-dimensional double extension of the metric Lie
  algebra and, indeed, bargmannian Lie algebras coincide with such
  double extensions, containing carrollian Lie algebras as an ideal
  and projecting to galilean Lie algebras.  This sets up a canonical
  correspondence between carrollian and galilean Lie algebras mediated
  by bargmannian Lie algebras.  This reformulation allows us to use
  the structure theory of metric Lie algebras to give a list of
  bargmannian, carrollian and galilean Lie algebras in the
  positive-semidefinite case.  We also characterise Lie groups
  admitting a bi-invariant (ambient) leibnizian structure.  Leibnizian
  Lie algebras extend the class of bargmannian Lie algebras and also
  set up a non-canonical correspondence between carrollian and
  galilean Lie algebras.
\end{abstract}
\dedicatory{În memoria Veronicăi Stanciu}
\thanks{EMPG-22-20, \href{https://orcid.org/0000-0002-9308-9360}{ORCID: 0000-0002-9308-9360}}
\maketitle
\tableofcontents

\section{Introduction}
\label{sec:introduction}

The study of non-lorentzian spacetime geometries is coming of age
(see, e.g., the recent reviews \cite{MR4422358,Bergshoeff:2022eog}),
yet there are still some simple questions which have not been asked
nor answered.  A natural first step when studying an unfamiliar
geometric structure is to find examples of such structures with lots
of symmetries.  A natural class of such examples are homogeneous
spaces and, in particular, Lie groups with bi-invariant structures.
Bi-invariance is typically quite strong and this often allows one to
classify them or at least to characterise them in linear algebraic
terms.  This is the case, for example, with Lie groups admitting a
bi-invariant metric (of any signature), which were studied by Medina
\cite{MedinaLorentzian} and characterised in terms of their Lie
algebras by Medina and Revoy \cite{MedinaRevoy} (see also
\cite{Figueroa-OFarrill:1994liu,Figueroa-OFarrill:1995opp,MR2205075}).

The three protagonists of today's tale are Bargmann, Carroll and
Galilei.  They may be used to label Lie groups, Lie algebras,
homogeneous spaces and also Cartan geometries (here, equivalently,
$G$-structures).  It turns out that in all of these settings, objects
with these names sit in relation to each other in a way which suggests
a correspondence (loosely, a duality) between Carroll and Galilei
mediated by Bargmann.  This correspondence was first pointed out in a
geometric context in \cite{Duval:2014uoa}, but before describing this
result let us set the stage by discussing the Lie algebras themselves.

The Carroll and Galilei Lie algebras are examples of kinematical Lie
algebras
\cite{Bacry:1968zf, Bacry:1986pm}.  In spatial
dimension $n$, they are spanned by $\left(L_{ab}, B_a, P_a, H\right)$,
where $L_{ab}=-L_{ba}$ span a Lie subalgebra isomorphic to $\so(n)$
under which $B_a, P_a$ are vectors and $H$ a scalar.  These conditions
translate into the following Lie brackets which are common to all
kinematical Lie algebras
\begin{equation}
  \label{eq:kla-generic}
  \begin{split}
    [L_{ab},L_{cd}] &= \delta_{bc} L_{ad} - \delta_{ac} L_{bd} -  \delta_{bd} L_{ac} + \delta_{bd} L_{ac} \\
    [L_{ab}, B_b] &= \delta_{bc} B_a - \delta_{ac} B_b\\
    [L_{ab}, P_b] &= \delta_{bc} P_a - \delta_{ac} P_b\\
    [L_{ab}, H] &= 0.
  \end{split}
\end{equation}
The kinematical Lie algebra where all other brackets vanish is called
the \textbf{static} kinematical Lie algebra and denoted $\s$.  All 
kinematical Lie algebras are deformations of $\s$ \cite{Figueroa-OFarrill:2017ycu,
  Figueroa-OFarrill:2017tcy, Andrzejewski:2018gmz}.  The subalgebra
$\s_0$ of $\s$ spanned by $\left( L_{ab}, B_a, P_a \right)$ will play
a rôle in our discussion.

The Carroll Lie algebra $\c$ is a central extension of $\s_0$, with
$H$ the central element and additional nonzero bracket
\begin{equation}
  \label{eq:carroll}
  [B_a, P_b] = \delta_{ab} H.
\end{equation}
In contrast, the Galilei Lie algebra $\g$ is an
``extension-by-derivation'' of $\s_0$.  The derivation is $\ad_H =
[H,-]$ and is defined by $\ad_H(L_{ab}) = \ad_H(P_a) = 0$ and
$\ad_H(B_a) = -P_a$, resulting in the additional nonzero bracket
\begin{equation}
  \label{eq:galilei}
  [H,B_a] = - P_a.
\end{equation}

We may summarise these observations in the diagrammatical language of
(short) exact sequences of Lie algebras as follows:
\begin{equation}
  \label{eq:pre-diag}
  \begin{tikzcd}
    & 0 \arrow[d] & & & \\
    & \RR \arrow[d] & & & \\
    & \c \arrow[d] & & & \\
    0 \arrow[r] & \s_0 \arrow[r] \arrow[d] & \g \arrow[r] & \RR \arrow[r] & 0\\
    & 0 & & & \\
  \end{tikzcd}
\end{equation}
The exact row says that $\g$ is an extension-by-derivation of $\s_0$,
whereas the exact column says that $\c$ is a one-dimensional extension
of $\s_0$.  The diagram does not fix the brackets uniquely, since
$\s_0$ admits many derivations and also other one-dimensional
extensions, central or not.

The Bargmann Lie algebra $\b$ is a central extension of the Galilei
Lie algebra $\g$ with additional generator $M$ and additional bracket
\begin{equation}
  \label{eq:bargmann}
  [B_a, P_b] = \delta_{ab} M,
\end{equation}
which is reminiscent of the bracket \eqref{eq:carroll} with $M$
playing the rôle of $H$.  Under this relabelling of bases, we see that
the Bargmann Lie algebra $\b$ is an extension-by-derivation of the
Carroll Lie algebra $\c$.  The derivation is $\ad_H = [H,-]$ where
again $\ad_H$ annihilates $L_{ab}, P_a, M$ and its action on $B_a$ is
given by the bracket \eqref{eq:galilei}.

This allows us to complete the diagram \eqref{eq:pre-diag} to the
following commutative diagram of Lie algebras:
\begin{equation}
  \label{eq:comm-diag}
  \begin{tikzcd}
    & 0 \arrow[d]                                   &  0  \arrow[d]                                 &                                        & \\
    & \mathbb{R} \arrow[r,equal] \arrow[d]         & \mathbb{R}  \arrow[d]                        &                                        & \\
    0 \arrow[r] & \c \arrow[r] \arrow[d] & \b \arrow[r] \arrow[d] & \mathbb{R} \arrow[r] \arrow[d,equal] & 0 \\
0 \arrow[r] & \s_0 \arrow[r] \arrow[d]            & \g \arrow[r] \arrow[d] & \mathbb{R} \arrow[r] & 0 \\
            & 0                                             &  0                                            &                                        & \\
  \end{tikzcd}
\end{equation}
which will recur in other contexts in this work with other Lie
algebras playing the rôles of $\s_0$, $\c$, $\g$ and $\b$.  In fact,
this very diagram has already appeared in
\cite[App.~B]{Figueroa-OFarrill:2022tlf} in the context of the
hamiltonian description of particle dynamics on the (spatially
isotropic) homogeneous galilean spacetimes classified in
\cite{Figueroa-OFarrill:2018ilb}.  In that context, the Lie algebras
which play the rôles of $\s_0$, $\c$, $\g$ and $\b$ are
infinite-dimensional.  The two exact rows in the the diagram
\eqref{eq:comm-diag} say that $\b$ (resp. $\g$) is an
extension-by-derivation of $\c$ (resp. $\s_0$), whereas the two exact
columns say that $\b$ (resp. $\c$) is a one-dimensional extension of
$\g$ (resp. $\s_0$).  We could say, borrowing the terminology from the
theory of metric Lie algebras, that the Bargmann Lie algebra $\b$ is a
\emph{double extension}\footnote{To be clear, this proposal expands
  the definition of a double-extension, transcending its origin in
  the context of metric Lie algebras, to a more general notion in
  which a double extension is the composition of a one-dimensional
  extension with an extension-by-derivation; at least in those cases
  where these two operations commute.} of $\s_0$.  This is more
than a mere analogy and we will see that this is exactly right for Lie
algebras admitting ad-invariant bargmannian structures (to be defined
below).

The Bargmann Lie algebra $\b$ can also be defined as the subalgebra of
the Poincaré Lie algebra $\iso(n+1,1)$ in one dimension higher which
centralises a null translation.  Since the Poincaré translations
commute, they are contained in $\b$ and hence the (simply-connected)
Bargmann Lie group acts transitively on Minkowski spacetime
$\MM^{n+2}$.  This fact underlies the geometric Carroll/Galilei
duality in \cite{Duval:2014uoa}.

Choosing a Witt frame $(\be_a,\be_+,\be_-)$ for Minkowski spacetime
$\MM^{n+2}$ we may express the generators of the Poincaré algebra as
$\left(L_{ab}, L_{+a}, L_{-a}, L_{+-}, P_a, P_+, P_-\right)$.  The
centraliser of $P_+$ is spanned by
$\left(L_{ab}, L_{+a}, P_a, P_+, P_-\right)$ and it is isomorphic to
the Bargmann Lie algebra $\b$, with $P_+$ playing the rôle of $M$.  If
$X$ is an element of the Poincaré Lie algebra, we let $\xi_X$ denote
the corresponding Killing vector field on Minkowski spacetime.  The
null vector field $\xi_{P_+}$ on Minkowski spacetime is not just
Killing but actually parallel.  It defines a distribution
$\xi_{P_+}^\perp \subset T\MM$ which is integrable.  The leaves of the
corresponding foliation are null hypersurfaces and are copies of the
Carroll spacetime $\zC^{n+1}$.  Indeed, they are homogeneous spaces of
the normal subgroup with Lie algebra the ideal of $\b$ spanned by
$\left(L_{ab}, L_{+a}, P_a, P_+\right)$, which is isomorphic to the
Carroll Lie algebra $\c$, with $P_+$ playing the rôle of $H$.  On the
other hand, the null reduction \cite{PhysRevD.31.1841,Julia:1994bs} of
$\MM$ by the one-parameter subgroup generated by $\xi_{P_+}$ is
isomorphic to Galilei spacetime $\zG^{n+1}$.  Indeed, the quotient is
homogeneous under the Lie group whose Lie algebra is the quotient of
$\b$ by the line spanned by $P_+$, which is isomorphic to the Galilei
Lie algebra $\g$.  This results in the following suggestive diagram
\begin{equation}
  \label{eq:car-gal-dual-spacetimes}
  \begin{tikzcd}
    \zC^{n+1} \arrow[r,hook] & \MM^{n+2} \arrow[r,two heads] & \zG^{n+1},
  \end{tikzcd}
\end{equation}
which is the fundamental example of the geometric Carroll/Galilei
duality in \cite{Duval:2014uoa}.

This duality seems to be broken when we take all (spatially isotropic)
homogeneous kinematical spacetimes into consideration.  In the
classification of \cite{Figueroa-OFarrill:2018ilb} there are three
other homogeneous carrollian spacetimes besides the Carroll spacetime:
the carrollian limits $\zdSC$ of de Sitter and $\zAdSC$ of
anti~de~Sitter spacetimes, and the lightcone $\zLC$, which can be
realised as null hypersurfaces in de Sitter, anti~de~Sitter and
Minkowski spacetimes, respectively
\cite{Figueroa-OFarrill:2018ilb,Morand:2018tke}.  On the other hand,
there are two one-parameter families of homogeneous galilean
spacetimes.  The galilean limit $\zdSG$ of de~Sitter spacetime is one
point ($\gamma = -1$) in a continuum $\zdSG_\gamma$, for $\gamma \in
[-1,1]$, of galilean spacetimes.  Similarly, the galilean limit
$\zAdSG$ of anti~de~Sitter spacetime is one point $(\chi = 0)$ in a
continuum $\zAdSG_\chi$, for $\chi \geq 0$, of galilean spacetimes.
These two continua have a point in common, since $\lim_{\gamma \to 1}
\zdSG_\gamma = \lim_{\chi \to \infty} \zAdSG_\chi$.  It is often
convenient to describe homogeneous spaces infinitesimally via their
Klein pairs.  The above homogeneous galilean spacetimes have Klein
pairs $(\g_{\alpha,\beta}, \h)$ where $\g_{\alpha,\beta}$ is the
kinematical Lie algebra with additional nonzero brackets
\begin{equation}
  [H,B_a] = - P_a \qquad\text{and}\qquad [H,P_a] = \alpha B_a + \beta P_a,
\end{equation}
for some $(\alpha,\beta) \in \RR^2$ and $\h$ is the subalgebra spanned
by $\left(L_{ab},B_a\right)$.  The parameters $\gamma$ and $\chi$
labeling the continua of galilean spacetimes serve as convenient
parametrisations for the equivalence classes of pairs $(\alpha,\beta)$
under the relation $(\alpha,\beta) \sim (s^2 \alpha, s \beta)$ for any
nonzero $s \in \RR$.  Just like the Galilei Lie algebra $\g =
\g_{0,0}$, the Lie algebra $\g_{\alpha,\beta}$ also admits a
one-dimensional extension $\b_{\alpha,\beta}$ with additional brackets
\begin{equation}
  \label{eq:barg-ext}
  [B_a, P_b] = \delta_{ab} M \qquad\text{and}\qquad [H,M] = \beta M,
\end{equation}
which is central if and only if $\beta = 0$.  In that case, we may
distinguish three cases: $\alpha=0$, corresponding to the Galilei Lie
algebra, $\alpha >0$ and $\alpha < 0$, corresponding to the two
Newton--Hooke Lie algebras.

The Lie algebra $\b_{\alpha,\beta}$ can also be described as an
extension-by-derivation of the Carroll Lie algebra $\c$.  Again we
need to relabel the generators $M \leftrightarrow H$ and the
derivation $\ad_H$ is defined by $\ad_H(L_{ab}) = 0$, $\ad_H(B_a) =
-P_a$, $\ad_H(P_a) = \alpha B_a + \beta P_a$ and $\ad_H(M) = \beta
M$.  In other words, we once again arrive at a commutative diagram
like \eqref{eq:comm-diag}
\begin{equation}
  \label{eq:comm-diag-too}
  \begin{tikzcd}
    & 0 \arrow[d]                                   &  0  \arrow[d]                                 &                                        & \\
    & \mathbb{R} \arrow[r,equal] \arrow[d]         & \mathbb{R}  \arrow[d]                        &                                        & \\
    0 \arrow[r] & \c \arrow[r] \arrow[d] & \b_{\alpha,\beta} \arrow[r] \arrow[d] & \mathbb{R} \arrow[r] \arrow[d,equal] & 0 \\
0 \arrow[r] & \s_0 \arrow[r] \arrow[d]            & \g_{\alpha,\beta} \arrow[r] \arrow[d] & \mathbb{R} \arrow[r] & 0 \\
            & 0                                             &  0                                            &                                        & \\
  \end{tikzcd}
\end{equation}

In this note we will exhibit yet other avatars of the above diagram.
Indeed, our main aim in this note is the classification of (connected,
simply-connected) Lie groups admitting bargmannian, carrollian or
galilean structures.  The Lie correspondence allows us to work at the
level of Lie algebras with ad-invariant bargmannian, carrollian or
galilean structures.  Without aiming to confuse the reader, we shall
refer to such Lie algebras as bargmannian, carrollian or galilean, for
short.  However please note that, paradoxically perhaps, the Bargmann
Lie algebra is not bargmannian, the Carroll Lie algebra is not
carrollian and the Galilei Lie algebra is not galilean.  Nevertheless
bargmannian, carrollian and galilean Lie algebra sit in relation to
each other just like the Bargmann, Carroll and Galilei Lie algebras;
namely, in a commutative diagram such as \eqref{eq:comm-diag}.
Neither should our bargmannian, carrollian and galilean Lie algebras
be confused with the transitive Lie algebras of kinematical spacetimes
with such structures.  Metric Lie algebras associated with kinematical
Lie algebras were studied in \cite{Matulich:2019cdo} in the context of
Chern--Simons theories of gravity.

This note is organised as follows.  In
Section~\ref{sec:barg-carr-galil-struct} we review the notions of
carrollian and galilean structures and their duality via bargmannian
manifolds and specialise to connected Lie groups admitting
bi-invariant such structures.  This allows us to localise at the
identity and discuss ad-invariant bargmannian, carrollian and galilean
structures on Lie algebras, which are studied in
Section~\ref{sec:lie-algebras}.  We will see that all three kinds of
Lie algebras are characterised in terms of the same data: namely, a
metric Lie algebra with a skew-symmetric derivation.  This is also the
data which defines a one-dimensional double extension of a metric Lie
algebra, which is a construction due to Medina and Revoy
\cite{MedinaRevoy} (see also
\cite{Figueroa-OFarrill:1994liu,Figueroa-OFarrill:1995opp}) which
results in a new metric Lie algebra.  This will allow us to identify
bargmannian Lie algebras precisely as one-dimensional double
extensions of metric Lie algebras.  In
Section~\ref{sec:classification} we specialise to strictly carrollian
and galilean structures (as opposed to the ``pseudo'' versions we
treated before) and classify the relevant Lie algebras.  In
Section~\ref{sec:leibnizian-groups} we discuss Lie groups with
bi-invariant leibnizian structures in the sense of
\cite{Bekaert:2015xua} and show that they too appear in a commutative
diagram such as \eqref{eq:comm-diag}.  Leibnizian groups (strictly)
extend the class of bargmannian Lie groups and they too mediate a
duality between carrollian and galilean Lie group, which is not
canonical and hence in principle different from the one discussed in
Section~\ref{sec:summary}.  Finally in Section~\ref{sec:conclusion} we
offer some concluding remarks.  There are two appendices: in
Appendix~\ref{sec:skew-symm-deriv} we give a proof of a result needed
in Section~\ref{sec:classification}; and in
Appendix~\ref{sec:once-more-with-indices} we write some of the Lie
algebraic structures in terms of a basis.

\section{Bargmannian, carrollian and galilean structures}
\label{sec:barg-carr-galil-struct}

Let $M$ be a finite-dimensional smooth manifold.  Recall that a
\textbf{galilean structure} \cite{MR334831} on $M$ consists of a
nowhere-vanishing one-form $\tau \in \Omega^1(M)$ and a symmetric
bivector field $\gamma \in \Gamma(\odot^2TM)$, which is everywhere
corank-$1$ as a field of bilinear forms on one-forms and such that its
radical is everywhere spanned by $\tau$.  Typically one demands that
$\gamma$ is positive-semidefinite, but one can also consider
\textbf{pseudo-galilean structures} where $\gamma$ has any signature.
Galilean structures are particular examples of $G$-structures
\cite{MR334831, Figueroa-OFarrill:2020gpr}.

Similarly, a \textbf{carrollian structure}
\cite{Henneaux:1979vn,Duval:2014uoa} on $M$ consists of a
nowhere-vanishing vector field $\kappa \in \eX(M)$ and a symmetric
$(0,2)$-tensor field $h \in \Gamma(\odot^2 T^*M)$ which is everywhere
of corank-$1$ and such that the characteristic distribution is generated
by $\kappa$.  Again one typically demands that $h$ is positive-semidefinite,
but one can also consider \textbf{pseudo-carrollian structures} where
$h$ has any signature. For example, the blow-up of spatial infinity of
Minkowski spacetime is a pseudo-carrollian manifold
\cite{Gibbons:2019zfs}, whereas the blow-up of either past or future
timelike infinity is carrollian \cite{Figueroa-OFarrill:2021sxz}.
Carrollian structures are also examples of $G$-structures
\cite{Figueroa-OFarrill:2020gpr}.

One may exhibit a correspondence between carrollian and galilean
structures \cite{Duval:2014uoa} by passing to a higher-dimensional
manifold $M$ with a \textbf{bargmannian structure}, namely an
indefinite metric $g$ together with a nowhere-vanishing null vector
field $\xi$.  Strictly speaking, bargmannian structures require the
metric $g$ to be lorentzian, but one can also have
\textbf{pseudo-bargmannian structures} where $g$ is only assumed to be
of indefinite signature.  In any case, the null vector field $\xi$
defines a distribution $\xi^\perp \subset TM$ which, if integrable,
foliates $M$ by null (since $\xi \in \xi^\perp$) hypersurfaces
$i : N \to M$ which inherit a carrollian structure $(\xi, i^*g)$.  If
$\xi$ is Killing, then the null reduction \cite{PhysRevD.31.1841,
  Julia:1994bs} defines (in the good cases) a fibration
$\pi : M \to N$ where $N$ inherits a galilean structure whose clock
one-form pulls back to the one-form $\xi^\flat$ dual to $\xi$ and
whose spatial cometric is induced by the inverse of $g$.

In this note we are interested in Lie groups admitting bi-invariant
carrollian, galilean or bargmannian structures; that is, triples
$(G,\kappa, h)$, $(G,\tau,\gamma)$ or $(G,g,\xi)$ where $G$ is a
connected Lie group and $\kappa$ and $h$, $\tau$ and $\gamma$ or $g$
and $\xi$ are invariant under Lie derivatives by both left- and
right-invariant vector fields.  Such tensor fields are determined
uniquely by their value at the identity, where they define tensors of
the Lie algebra $\g$ of $G$ which are invariant under the adjoint
representation of $G$ on $\g$.  Since we assume assume that $G$ is
connected, it will be enough to demand that the tensors are
ad-invariant; that is, invariant under the adjoint action of $\g$ on
itself.

\section{Bargmannian, carrollian and galilean Lie algebras}
\label{sec:lie-algebras}

In this section we characterise Lie algebras admitting ad-invariant
bargmannian, carrollian or galilean structures.  They turn out to be
intimately linked to metric Lie algebras, which we review first.

\subsection{Metric Lie algebras and double extensions}
\label{sec:metric-lie-algebras}

Let $\g$ be a finite-dimensional real Lie algebra.  We say that $\g$
is a \textbf{metric Lie algebra} if it admits an ad-invariant inner
product $\left<-,-\right>$; that is, one for which
\begin{equation}
  \left<[W,X],Y\right> = - \left<X, [W,Y]\right>
\end{equation}
for all $W,X,Y \in \g$.  Medina and Revoy \cite{MedinaRevoy}
introduced the notion of a double extension, as a method of
constructing new metric Lie algebras from old.  In this note we will
only need a special case of this construction: namely, a
one-dimensional double extension.

The ingredients for a one-dimensional double extension are a metric
Lie algebra $\g_0$ with ad-invariant scalar product
$\left<-,-\right>_0$ and a skew-symmetric derivation $D_0$ of $\g_0$;
that is, for all $X,Y \in \g_0$,
\begin{equation}
  D_0 [X,Y]_0 = [D_0 X,Y]_0 + [X,D_0 Y]_0 \qquad\text{and}\qquad \left<D_0
    X, Y\right>_0 = - \left<X, D_0 Y\right>_0,
\end{equation}
where $[-,-]_0$ is the Lie bracket in $\g_0$.  On the vector space $\g
= \g_0 \oplus \RR D \oplus \RR Z$, where $D$ and $Z$ are two
additional generators, we define the brackets
\begin{equation}
  \label{eq:double-ext}
  \begin{split}
    [X,Y] &= [X,Y]_0 + \left<D_0 X, Y\right>_0 Z\\
    [D,X] &= D_0 X\\
    [Z,-] &= 0,
  \end{split}
\end{equation}
for all $X,Y \in \g_0$.  We can define a scalar product on $\g$ by
\begin{equation}
  \left<X,Y\right> = \left<X,Y\right>_0, \qquad \left<X,D\right> =
  \left<X,Z\right> = 0, \qquad \left<D,Z\right> =1
  \qquad\text{and}\qquad \left<D,D\right> \in \RR,
\end{equation}
for all $X,Y \in \g_0$, which can be seen to be ad-invariant.  We may
(and will) perform a Lie algebra automorphism
$D \mapsto D - \tfrac12 \left<D,D\right> Z$ which allows us to set
$\left<D,D\right>=0$.  The Lie algebra $\g$ is said to be the
\textbf{double extension} of $\g_0$ relative to the skew-symmetric
derivation $D_0$.  This construction is due to Medina and Revoy
\cite{MedinaRevoy}, refined and applied in the context of
two-dimensional conformal field theory in
\cite{Figueroa-OFarrill:1994liu,Figueroa-OFarrill:1995opp}, and
further refined by Kath and Olbrich \cite{MR2205075}.

\subsection{Bargmannian Lie algebras}
\label{sec:bargm-lie-algebr}

A metric Lie algebra $(\g,\left<-,-\right>)$ is said to be
\textbf{bargmannian} if there is a nonzero null $Z \in \g$, which is
central.  This of course requires $\left<-,-\right>$ to be indefinite,
yet still nondegenerate.  We will show that bargmannian Lie algebras
are one-dimensional double-extensions of metric Lie algebras.

\begin{proposition}
  Every bargmannian Lie algebra $\g$ is isomorphic to a double
  extension of a metric Lie algebra $\g_0$ by a skew-symmetric
  derivation.
\end{proposition}

\begin{proof}
  Let $\g$ be a bargmannian Lie algebra with scalar product
  $\left<-,-\right>$ and central null element $Z \in \g$.  The
  one-dimensional isotropic ideal spanned by $Z$ is clearly minimal
  and hence by the structure theorem of Medina and Revoy
  \cite{MedinaRevoy} (see also \cite[§3]{Figueroa-OFarrill:1995opp})
  the Lie algebra $\g$ is a double extension of the metric Lie algebra
  $\g_0 := Z^\perp/\RR Z$, which inherits an inner product from that
  of $\g$, by a skew-symmetric derivation defined as follows.  Let
  $D \in \g$ be such that $\left<D,Z\right>=1$.  Then $\ad_D$
  preserves $Z^\perp$ since $Z^\perp$ is an ideal of $\g$ and
  $[D,Z] = 0$ since $Z$ in central, hence $\ad_D$ induces a derivation
  $D_0$ of $\g_0$ as follows: if $X \in Z^\perp$ and letting
  $\Xbar \in \g_0$ denote its projection, we define the derivation
  $D_0$ by $D_0 \Xbar = \overline{[D,X]}$, which is well defined.  It
  is clearly a derivation and also skew-symmetric under the inner
  product $\left<-,-\right>_0$ on $\g_0$ defined by
  \begin{equation}
    \left<\Xbar, \Ybar\right>_0 = \left<X,Y\right>.
  \end{equation}
  The fact that $D_0$ is skew-symmetric follows from the fact that so
  is $\ad_D$:
  \begin{align*}
    \left<D_0 \Xbar, \Ybar \right>_0 &= \left<\overline{[D,X]}, \Ybar\right>_0\\
                                                &= \left<[D,X], Y\right>\\
                                                &= - \left<X, [D,Y]\right>\\
                                                &= - \left<\Xbar, \overline{[D,Y]}\right>_0\\
                                                &= - \left<\Xbar, D_0 \Ybar\right>_0.
  \end{align*}
  As vector spaces, $\g = Z^\perp
  \oplus \RR D$, where $Z^\perp = \g_0 \oplus \RR Z$ and the Lie
  brackets are given by
  \begin{equation}
    \left[ X,Y \right] = \left[ X,Y \right]_0 + \left<D_0 X,
      Y\right>_0 Z, \qquad \left[ D,X \right] = D_0 X \qquad\text{and}\qquad
    \left[ Z,- \right] = 0,
\end{equation}
for all $X,Y \in \g_0$, where we have adorned the bracket and the
inner product of $\g_0$ with a subscript ${}_0$.  The inner product
can be brought to the form
\begin{equation}
  \left<X,Y\right> = \left<X,Y\right>_0,\qquad
  \left<X,D\right> = 0,\qquad
  \left<X,Z\right> = 0,\qquad
  \left<D,Z\right> = 1,\qquad\text{and}\qquad
  \left<D,D\right> = 0,
\end{equation}
for all $X,Y \in \g_0$.  In other words, comparing with
Section~\ref{sec:metric-lie-algebras}, we see that $\g$ is the double
extension of $\left( \g_0,\left<-,-\right>_0\right)$ by the
skew-symmetric derivation $D_0$.
\end{proof}

Notice that if the inner product on $\g_0$ has signature $(p,q)$, the
one on $\g$ has signature $(p+1,q+1)$.  Hence if $\g$ is lorentzian,
then $\g_0$ must have a positive-definite inner product.  In that
case, it follows from the structure theorem of Medina and Revoy
\cite{MedinaRevoy} (see also \cite{Figueroa-OFarrill:1995opp}) that
$\g_0$ is isomorphic to an orthogonal direct sum of compact simple Lie
algebras $\s_i$ (each with a negative multiple of the Killing form)
and an abelian Lie algebra $\a$ with a euclidean inner product.  We
will use this in Section~\ref{sec:classification} in order to classify
(strictly) carrollian and galilean Lie algebras.

\subsection{Carrollian Lie algebras}
\label{sec:car-lie-alg}

Let $\g$ be a finite-dimensional real Lie algebra.  By an ad-invariant
\textbf{carrollian structure} on $\g$ we mean a pair $(Z,h)$
consisting of a (nonzero) central element $Z \in \g$ and an
ad-invariant $h \in \odot^2\g^*$ whose radical is one-dimensional and
spanned by $Z$.

\begin{proposition}
  Every carrollian Lie algebra $\g$ is isomorphic to a one-dimensional
  central extension of a metric Lie algebra $\g_0$.
\end{proposition}

\begin{proof}
  Let $(\g, Z, h)$ be a carrollian Lie algebra and let $\g_0 =
\g/\RR Z$.  Then $\g$ is a central extension of $\g_0$, which we may
summarise by the following short exact sequence of Lie algebras:
  \begin{equation}
    \label{eq:car-la-ses}
    \begin{tikzcd}
      0 \arrow[r] & \RR Z \arrow[r] & \g \arrow[r] & \g_0 \arrow[r] & 0.
    \end{tikzcd}
  \end{equation}
  Moreover, $h$ induces an inner product $\left<-,-\right>_0$ on $\g_0$ by
  \begin{equation}
    \left<\Xbar, \Ybar\right>_0 = h(X,Y),
  \end{equation}
  where $\Xbar \in \g_0$ is the image of $X \in \g$ under the
  canonical map $\g \to \g/\RR Z$.  The inner product is
  well-defined since $h(Z,-)=0$.  Since $Z$ spans the radical,
  $\left<-,-\right>_0$ is nondegenerate and since $h$ is
  $\g$-invariant, $\left<-,-\right>_0$ is $\g_0$-invariant.  In other
  words, $\left(\g_0, \left<-,-\right>_0\right)$ is a metric Lie algebra.
\end{proof}

As vector spaces, $\g = \g_0 \oplus \RR Z$ and the Lie bracket is
given by
\begin{equation}
  \begin{split}
    [\Xbar,\Ybar] &= [\Xbar,\Ybar]_0 + \alpha(\Xbar,\Ybar)Z \\
    [Z,-] &= 0
  \end{split}
\end{equation}
for all $\Xbar,\Ybar \in \g_0$ and where $[-,-]_0$ is the Lie bracket on
$\g_0$.  It follows that $\alpha \in \wedge^2\g_0^*$ is a
$2$-cocycle.  Corresponding to a $2$-cocycle on any metric Lie algebra,
there is always a skew-symmetric derivation $D_0$, defined by
\begin{equation}\label{eq:deriv-from-cocycle}
  \left<D_0\Xbar,\Ybar\right>_0 = \alpha(\Xbar,\Ybar)
\end{equation}
for all $\Xbar,\Ybar \in \g_0$.  Since $\left<-,-\right>_0$ is nondegenerate,
$D_0$ is uniquely defined.  Since $\alpha$ is alternating, $D_0$ is
skew-symmetric:
\begin{equation}
  \left<D_0\Xbar,\Ybar\right>_0 = \alpha(\Xbar,\Ybar) = - \alpha(\Ybar,\Xbar) = - \left<\Xbar,D_0\Ybar\right>_0.
\end{equation}
Finally, the cocycle condition for $\alpha$, says that $D_0$ is a
derivation: for every $\Xbar,\Ybar,\Ubar \in \g_0$,
\begin{align*}
  \left<D_0[\Xbar,\Ybar]_0, \Ubar\right>_0 &= \alpha([\Xbar,\Ybar]_0,\Ubar)\\
                         &= \alpha(\Xbar,[\Ybar,\Ubar]_0) + \alpha(\Ybar,[\Ubar,\Xbar]_0)\\
                         &= \left<D_0\Xbar,[\Ybar,\Ubar]_0\right>_0 + \left<D_0\Ybar,[\Ubar,\Xbar]_0\right>_0\\
                         &= \left<[D_0\Xbar,\Ybar]_0,\Ubar\right>_0 - \left<[D_0\Ybar,\Xbar]_0,\Ubar\right>_0\\
                         &= \left<[D_0\Xbar,\Ybar]_0,\Ubar\right>_0 + \left<[\Xbar,D_0\Ybar]_0,\Ubar\right>_0,
\end{align*}
so that $D_0[\Xbar,\Ybar]_0 = [D_0\Xbar,\Ybar]_0 + [\Xbar,D_0\Ybar]_0$.

Let $\ghat$ denote the one-dimensional double-extension of $\g_0$
relative to the skew-symmetric derivation $D_0$, as discussed in
Section~\ref{sec:metric-lie-algebras}.  Then the carrollian Lie
algebra $\g$ embeds into $\ghat$ as the ideal $Z^\perp$,
resulting in the following short exact sequence of Lie algebras:
\begin{equation}
  \label{eq:car-2ext-la-ses}
  \begin{tikzcd}
    0 \arrow[r] & \g \arrow[r] & \ghat \arrow[r] & \RR D \arrow[r] & 0.
  \end{tikzcd}
\end{equation}

\subsection{Galilean Lie algebras}
\label{sec:galil-lie-algebr}

Let $\g$ be again a finite-dimensional real Lie algebra.  By an
ad-invariant \textbf{galilean structure} on $\g$ we mean a nonzero
ad-invariant $\tau \in \g^*$ and an ad-invariant $\gamma \in
\odot^2\g$, defining a symmetric bilinear form on $\g^*$ which has a
one-dimensional radical spanned by $\tau$.

\begin{proposition}
  Every galilean Lie algebra is an ``extension by skew-symmetric
  derivation'' of a metric Lie algebra $\g_0$.
\end{proposition}

\begin{proof}
  Let $(\g, \tau,\gamma)$ be a galilean Lie algebra.  Since $\tau$
  is ad-invariant, it annihilates brackets ($\tau([X,Y]) = 0$) and
  hence $\g_0 := \ker \tau$ is an ideal of $\g$.  We thus get a short
  exact sequence of Lie algebras
  \begin{equation}
    \label{eq:gal-la-ses}
    \begin{tikzcd}
      0 \arrow[r] & \g_0 \arrow[r] & \g \arrow[r,"\tau"] & \RR \arrow[r] & 0.
    \end{tikzcd}
  \end{equation}
  This sequence always splits: let $D \in \g$ be any element with
  $\tau(D) = 1$.  Then since $\tau$ annihilates brackets, $\ad_D$
  is a derivation of $\g_0$.  Dualising the above sequence we see that
  $\g_0^* \cong \g^*/\RR\tau$.  The tensor $\gamma$ defines an inner
  product $\gamma_0$ on $\g_0^*$.  Since $\gamma$ is $\g$-invariant,
  so is $\gamma_0$.  The inverse of $\gamma_0$ is an invariant inner
  product $\left<-,-\right>_0$ on $\g_0$ relative to which $\ad_D$ is
  skewsymmetric.  In summary, a galilean Lie algebra is an ``extension
  by skew-symmetric derivation'' of a metric Lie algebra.
\end{proof}

So just as in the case of carrollian Lie algebras, the underlying data
is again a metric Lie algebra and a skew-symmetric derivation.  If we
again let $\ghat$ denote the (one-dimensional) double extension
of $\g_0$ by the skew-symmetric derivation $\ad_D$, we have that the
galilean Lie algebra $\g$ is a quotient of $\ghat$ by the central
line $\RR Z$, resulting the short exact sequence
\begin{equation}
  \label{eq:gal-2ext-la-ses}
  \begin{tikzcd}
    0 \arrow[r] & \RR Z \arrow[r] & \ghat \arrow[r] & \g \arrow[r] & 0.
  \end{tikzcd}
\end{equation}

\subsection{Summary}
\label{sec:summary}

Let $\left( \g_0, \left<-,-\right>_0 \right)$ be a metric Lie algebra
and $D_0$ a skew-symmetric derivation.  This data allows us to define
three other Lie algebras:
\begin{enumerate}
\item a bargmannian Lie algebra $\widehat g$, the one-dimensional
  double extension of $\g_0$ by $D_0$;
\item a carrollian Lie algebra $\g_{\text{car}}$ which is a central
  extension of $\g_0$ with cocycle given by $D_0$ via
  \eqref{eq:deriv-from-cocycle} or, equivalently, an ideal of
  $\widehat \g$; and
\item a galilean Lie algebra $\g_{\text{gal}}$ which is an extension
  by the skew-symmetric derivation $D_0$ of $\g_0$ or, equivalently, a
  quotient of $\widehat \g$.
\end{enumerate}
These Lie algebras fit into the following commutative diagram:
\begin{equation}
  \label{eq:comm-diag-three}
  \begin{tikzcd}
            & 0 \arrow[d]                                   &  0  \arrow[d]                                 &                                        & \\
            & \mathbb{R} \arrow[r,equal] \arrow[d]         & \mathbb{R}  \arrow[d]                        &                                        & \\
0 \arrow[r] & \g_{\text{car}} \arrow[r] \arrow[d] & \ghat \arrow[r] \arrow[d] & \mathbb{R} \arrow[r] \arrow[d,equal] & 0 \\
0 \arrow[r] & \g_0 \arrow[r] \arrow[d]            & \g_{\text{gal}} \arrow[r] \arrow[d] & \mathbb{R} \arrow[r] & 0 \\
            & 0                                             &  0                                            &                                        & \\
\end{tikzcd}
\end{equation}
where the two short exact rows are extensions-by-derivations and the
two short exact columns are central extensions.  This defines a
canonical correspondence between carrollian and galilean Lie algebras.
For example, we start with a carrollian Lie algebra
$(\g_{\text{car}},Z,h)$ given by the data $(\g_0, D_0)$.  This Lie
algebra is an ideal of the double extension $\widehat \g$ of $\g_0$ by
$D_0$ and then we define the galilean dual $\g_{\text{gal}}$ of
$\g_{\text{car}}$ to be the quotient of $\ghat$ by the ideal generated
by $Z$.  Conversely, let $(\g_{\text{gal}},\tau,\gamma)$ be a galilean
Lie algebra given by the data $(\g_0,D_0)$ and let $\ghat$ again
denote the double extension of $\g_0$ by $D_0$.  Then
$\g_{\text{gal}}$ is a quotient of $\ghat$ by a one-dimensional ideal
spanned by $Z$ and we define the carrollian dual $\g_{\text{car}}$ of
$\g_{\text{gal}}$ as the ideal $Z^\perp$ of $\ghat$.

In Section~\ref{sec:leibnizian-groups} we will discuss a non-canonical
correspondence between carrollian and galilean algebras which is
mediated by a leibnizian Lie algebra (to be defined below), but first
we will use the canonical correspondence just established to classify
strictly carrollian, galilean and bargmannian Lie algebras.

\section{Classification}
\label{sec:classification}

In this section we specialise to strict carrollian and galilean
structures, where $h$ and $\gamma$ are positive-semidefinite. This
means that the metric Lie algebra $\g_0$ has a positive-definite inner
product or, equivalently, that the bargmannian Lie algebra
$\ghat$ is lorentzian. Lorentzian Lie algebras were classified by
Medina \cite{MedinaLorentzian}, but we are interested only in those
which are isomorphic to one-dimensional double extensions of a
positive signature metric Lie algebra. Since double extensions always
make the inner product indefinite, the structure theorem says that
metric Lie algebras with positive-definite inner products are
necessarily orthogonal direct sums of (compact) simple Lie algebras
(with a negative multiple of the Killing form) and an abelian Lie
algebra with a (trivially invariant) euclidean inner product.
Therefore we will let
$\g_0 = \s_1 \oplus \cdots \oplus \s_k \oplus \a$, where the $\s_i$
are simple and $\a$ is abelian. The direct sums are orthogonal and the
inner product on $\s_i$ is given by $-\lambda_i \kappa_i$, where
$\lambda_i > 0$ and $\kappa_i$ is the Killing form of $\s_i$, which
for compact simple Lie algebras is negative-definite, hence the sign.
The inner product on $\a$ is any desired euclidean inner product.  Up
to an isomorphism, we can think of $\a = \RR^m$ for some $m$ and the
inner product being the standard euclidean inner product on $\RR^m$.

Next we determine the skew-symmetric derivations of $\g_0$.  The
following proposition, proved in Appendix~\ref{sec:skew-symm-deriv},
states that any skew-symmetric derivation of $\g_0$ is the sum of an
inner derivation of the semisimple part $\ses = \s_1 \oplus \cdots
\oplus \s_k$ and a skew-symmetric endomorphism of the abelian part
$\a$.

\begin{proposition}
  \label{prop:skew-sym-der}
  Let $D_0$ be a skew-symmetric derivation of $\g_0$.  Then
  \begin{equation*}
    D_0 = T + \sum_i \ad_{X_i},
  \end{equation*}
  where $T \in \so(\a)$ and $X_i \in \s_i$.
\end{proposition}

The double extension $\ghat$ of $\g_0$ by the skew-symmetric
derivation $D_0$ has underlying vector space
$\g_0 \oplus \RR D \oplus \RR Z$ and the brackets are given by
\eqref{eq:double-ext}.  We may apply the following general linear
transformation of the vector space: it is the identity on
$\g_0 \oplus \RR Z$ and maps $D \mapsto D - \sum_i X_i$.  In that
basis, $[D,X] = T(X_\a)$, where $X = X_\ses + X_\a$ with
$X_\ses \in \ses$ and $X_\a \in \a$. In other words, we may change
basis so that effectively $D_0 \in \so(\a)$.  Using the euclidean
inner product on $\a$, we may identify $D_0 \in \so(\a)$ with
$\omega \in \wedge^2\a^*$ where
$\omega(A,B) = \left<D_0 A, B\right>_0$ for all $A,B \in \a$.  Let
$\omega$ have rank $2\ell$ for some $\ell = 0,1,\dots,
\lfloor\frac{\dim\a}{2}\rfloor$.

If $\ell = 0$, then $D_0  = 0$ and
\begin{equation}
  \widehat \g = \g_0 \oplus \underbrace{\RR D \oplus \RR Z}_{\a_2},
\end{equation}
where $\a_2$ is a two-dimensional abelian Lie algebra with a
lorentzian inner product.

If $\ell > 0$, we may split $\a = \a_0 \oplus \a_1$, where
$\a_0 = \ker D_0$ and $\a_1 = \a_0^\perp$, and $\omega$ defines a
symplectic structure on $\a_1$.  Then the double extension is
\begin{equation}
  \widehat \g = \ses \oplus \a_0 \oplus \nw_{2\ell + 2},
\end{equation}
where $\nw_{2\ell + 2}$ is a lorentzian Nappi--Witten algebra
\cite{Nappi:1993ie,Sfetsos:1993rh,Sfetsos:1993na,Figueroa-OFarrill:1994liu}
of dimension $2\ell + 2$ with underlying vector space $\a_1 \oplus \RR
D \oplus \RR Z$ and brackets
\begin{equation}
  \label{eq:nappi-witten}
  [A,B] = \omega(A,B) Z\qquad\text{and}\qquad [D,A] = D_0 A
\end{equation}
for all $A,B \in \a_1$ and with lorentzian inner product
\begin{equation}
  \left<A,B\right> = \left<A,B\right>_0 , \qquad \left<D,Z\right> = 1
  \qquad\text{and}\qquad \left<D,D\right> =0,
\end{equation}
where the choice $\left<D,D\right>=0$ can be arrived at via a Lie
algebra automorphism $D \mapsto D - \tfrac12 \left<D,D\right> Z$.

The corresponding carrollian Lie algebra is the ideal $Z^\perp$ of
$\ghat$.  As a Lie algebra,
\begin{equation}
  \label{eq:carrollian}
  Z^\perp = \ses \oplus \a_0 \oplus \heis_{2\ell + 1}
\end{equation}
where the Heisenberg Lie algebra $\heis_{2\ell +1}$ has underlying
vector space $\a_1 \oplus \RR Z$ and brackets
\begin{equation}
  [A,B] = \omega(A,B) Z
\end{equation}
for all $A,B \in \a_1$ with $Z$ central.

Finally, the corresponding galilean Lie algebra is the quotient
$\ghat / \RR Z$ with underlying vector space $\ses \oplus \a_0
\oplus \a_1 \oplus \RR D$, where the bracket on $\a_1 \oplus \RR D$ is
given by
\begin{equation}
  [A,B] = 0 \qquad\text{and}\qquad [D,A] = D_0 A,
\end{equation}
for all $A, B \in \a_1$.

In summary, the ingredients to construct any bargmannian, carrollian
and galilean Lie algebras in diagram \eqref{eq:comm-diag-three} are as
follows:
\begin{itemize}
\item a compact semisimple Lie algebra $\ses$ with a choice of
  positive-definite ad-invariant inner product: a negative multiple of
  the Killing form on each of its simple factors;
\item a euclidean vector space $\a_0$, whose only invariant is the
  dimension; and
\item a euclidean vector space $\a_1$ with a symplectic structure
  $\omega$.
\end{itemize}
Any Lie algebra with an ad-invariant bargmannian, carrollian or
galilean structure can be constructed out of these ingredients.  The
metric Lie algebra $\ses \oplus \a_0$ is always an orthogonal summand
of these Lie algebras, so we may concentrate on the
euclidean/symplectic vector space $(\a_1, \left<-,-\right>, \omega)$.

The bargmannian Lie algebra is of Nappi--Witten type and has
underlying vector space $\a_1 \oplus \RR Z \oplus \RR D$ and brackets,
for all $A, B \in \a_1$, given by
\begin{equation}
  [A,B] = \omega(A,B) Z \qquad\text{and}\qquad [D,A] = D_0 A,
\end{equation}
with $Z$ central and where $D_0$ is defined by $\left<D_0 A, B\right>
= \omega(A,B)$.  The underlying inner product extends the one on
$\a_1$ by declaring $Z$ to be null and $\left<D,Z\right> =1$.  We may
always redefine $D \mapsto D - \tfrac12 \left<D,D\right> Z$, which is
a Lie algebra automorphism, to ensure that $\left<D,D\right> = 0$.

The carrollian Lie algebra is a Heisenberg Lie algebra with underlying
vector space $\a_1 \oplus \RR Z$ and brackets
\begin{equation}
  [A,B] = \omega(A,B) Z,
\end{equation}
with $Z$ central.

The galilean Lie algebra is an extension of the abelian Lie algebra
$\a_1$ by the derivation $D_0$.  It has underlying vector space $\a_1
\oplus \RR D$ and the only nonzero bracket is the action $[D,A]  =
D_0 A$, for all $A \in \a_1$, of the derivation $D_0$.

In Appendix~\ref{sec:once-more-with-indices} we describe these Lie
algebras in a basis.

Finally, we tackle the isomorphism problem. Which triples  $(\a_1,
\left<-,-\right>, \omega)$ correspond to non-isomorphic Lie algebras.
Clearly if two such triples are in the same orbit of $\GL(\a_1)$, then
the resulting Lie algebras (bargmannian, carrollian and galilean) are
isomorphic.  Hence we need to classify triples $(\a_1,
\left<-,-\right>, \omega)$ up to the action of $\GL(\a_1)$.  We may
use $\GL(\a_1)$ to relate any two scalar products $\left<-,-\right>$,
leaving the freedom to act with the stabiliser $\Ort(\a_1)$ of the
scalar product on the symplectic form.  In other words, which are the 
possible symplectic forms $\omega$ on $\a_1$ up to orthogonal
transformations?  This is the same problem as classifying $D_0 \in
\so(\a_1)$ up to the adjoint action of $\Ort(\a_1)$.  Using
$\SO(\a_1)$ we may always conjugate to a Cartan subalgebra, or said
differently, we can skew-diagonalise $D_0$ to a matrix of the form
\begin{equation}
  \begin{pmatrix}
    \zero & \mu_1 & & & \\
    -\mu_1 & \zero & & & \\
    & & \ddots & &\\
    & & & \zero & \mu_n\\
    & & & -\mu_n & \zero
  \end{pmatrix}
\end{equation}
where $\mu_i$ are nonzero and ordered.  Conjugating by the full
orthogonal group we may assume that $\mu_1 \geq \mu_2 \geq \dots \geq
\mu_n > 0$, where $\dim \a_1 = 2n$.

We may summarise this discussion as follows.

\begin{theorem}
  The isomorphism classes of bargmannian (and hence also carrollian
  and galilean) Lie algebras are parametrised by the following data:
  \begin{itemize}
  \item a compact Lie algebra with a choice of positive-definite
    ad-invariant scalar product (i.e., the orthogonal direct sum of a
    semisimple and an abelian Lie algebras); and
  \item $(\mu_1,\dots,\mu_n) \in \RR^n$ with $\mu_1 \geq \mu_2 \geq
    \dots \geq \mu_n > 0$, which are the skew-eigenvalues of a
    symplectic form on the euclidean space $\mathbb{E}^{2n}$.
  \end{itemize}
  The Lie brackets and the ad-invariant bargmannian, carrollian and
  galilean structures associated to that data can be constructed as
  explained above.
\end{theorem}

In other words, the list of bargmannian Lie algebras, from which we
can then obtain all carrollian and galilean dual pairs is parametrised
by compact semisimple Lie algebras with a choice of positive-definite
invariant scalar product, a euclidean space of dimension $m\geq 0$ and
$n$ positive real numbers $\mu_1 \geq \dots \geq \mu_n > 0$.

\section{Leibnizian Lie groups}
\label{sec:leibnizian-groups}

In this section we characterise Lie groups with bi-invariant
leibnizian structures.  Such Lie groups extend the class of
bargmannian Lie groups and we will see that they also contain
carrollian Lie group as a normal subgroup and they quotient to a
galilean Lie group.  At the Lie algebra level, they too sit in
relation to each other in a way described by a diagram isomorphic to
\eqref{eq:comm-diag}.  Hence they define a non-canonical
correspondence between carrollian and galilean Lie algebras.

\subsection{Leibnizian structures}
\label{sec:leibn-struct}

In \cite{Bekaert:2015xua} Bekaert and Morand introduced the notion of
an (ambient) leibnizian structure on a manifold $M$.  This is a triplet
$(\xi,\psi,h)$ where $\xi \in \eX(M)$ is a nowhere-vanishing
vector field, $\psi \in \Omega^1(M)$ is a nowhere-vanishing one-form
with $\psi(\xi) = 0$, and $h \in \Gamma(\odot^2(\ker\psi)^*)$ is
a corank-$1$, positive-semidefinite, bilinear form on $\ker\psi \subset
TM$ whose radical is spanned by $\xi$.  Every bargmannian manifold
$(M,g,\xi)$ has a natural leibnizian structure given by $\psi =
\xi^\flat$ and $h = \left.g\right|_{\xi^\perp = \ker \xi^\flat}$.

A natural question is: \emph{which Lie groups admit a bi-invariant
  leibnizian structure?}  For the same price, we may (and will)
actually consider pseudo-leibnizian structures, where $h$ is not
necessarily positive-semidefinite, but still corank-$1$ on
$\ker\psi$.  We shall refer to such Lie groups and their Lie algebras
as leibnizian.

\subsection{Leibnizian Lie algebras}
\label{sec:leibn-lie-algebr}

\begin{definition}
  A Lie algebra $\g$ is said to be \textbf{leibnizian} if it admits a
  nonzero central element $Z \in \g$, a nonzero ad-invariant $\psi \in
  \g^*$ (i.e., $\left.\psi\right|_{[\g,\g]} = 0$) with $\psi(Z) = 0$
  and an ad-invariant $h \in \odot^2(\ker \psi)^*$ of corank $1$ whose
  radical is spanned by $Z$.
\end{definition}

It is immediate that if $(\g,Z,\psi,h)$ is leibnizian, then
$(\ker\psi, Z,h)$ is carrollian.  As in Section~\ref{sec:car-lie-alg},
$\g_0 = \ker\psi/\RR Z$ is a metric Lie algebra.  Let $\pi: \ker\psi
\to \g_0$ send $X \mapsto \Xbar$.  Then the Lie bracket and inner
product on $\g_0$ are defined by
\begin{equation*}
  [\Xbar,\Ybar]_0 = \overline{[X,Y]} \qquad\text{and}\qquad
  \left<\Xbar,\Ybar\right>_0 = h(X,Y).
\end{equation*}
As before, $\ker\psi$ is a central extension of the metric Lie algebra
$\g_0$
\begin{equation}
  \begin{tikzcd}
    0 \arrow[r] & \RR Z \arrow[r] & \ker \psi \arrow[r,"\pi"] & \g_0 \arrow[r] & 0,
  \end{tikzcd}
\end{equation}
and the central extension defines a skew-symmetric derivation $D_0$ of
$\g_0$ as in equation \eqref{eq:deriv-from-cocycle}.

Similarly, if $(\g,Z,\psi,h)$ is leibnizian, let $\gbar:= \g/\RR Z$,
so that $\g$ is a central extension of $\gbar$:
\begin{equation}
  \begin{tikzcd}
    0 \arrow[r] & \RR Z \arrow[r] & \g \arrow[r,"\overline\pi"] & \gbar \arrow[r] & 0,
  \end{tikzcd}
\end{equation}
where $\overline\pi: \g \to \gbar$ sends $X \mapsto \Xbar$.  The map
$\pi: \ker\psi \to \g_0$ above is the restriction of $\overline\pi$ to
the ideal $\ker\psi$, which explains why we use the same notation.
Since $\ker\psi$ is a codimension-$1$ ideal of $\g$, $\g$ is an
extension-by-derivation of $\ker\psi$:
\begin{equation}\label{eq:carr-in-g}
  \begin{tikzcd}
    0 \arrow[r] & \ker\psi \arrow[r] & \g \arrow[r,"\psi"] & \RR \arrow[r] & 0
  \end{tikzcd}
\end{equation}
and hence it follows that $\g_0$ is a codimension-$1$ ideal of $\gbar$,
so that $\gbar$ is an extension by derivation of $\g_0$:
\begin{equation}\label{eq:metric-in-gbar}
  \begin{tikzcd}
    0 \arrow[r] & \g_0 \arrow[r] & \gbar \arrow[r,"\psibar"] & \RR \arrow[r] & 0,
  \end{tikzcd}
\end{equation}
where $\psibar(\Xbar) = \psi(X)$.

In summary, we have again an instance of our favourite diagram:
\begin{equation}
  \label{eq:comm-diag-four}
  \begin{tikzcd}
    & 0 \arrow[d]                                   &  0  \arrow[d]                                 &                                        & \\
    & \mathbb{R}Z \arrow[r,equal] \arrow[d]         & \mathbb{R}Z  \arrow[d]                        &                                        & \\
    0 \arrow[r] & \ker\psi \arrow[r] \arrow[d,"\pi"] & \g \arrow[r,"\psi"] \arrow[d,"\overline\pi"] & \mathbb{R} \arrow[r] \arrow[d,equal] & 0 \\
    0 \arrow[r] & \g_0 \arrow[r] \arrow[d]            & \gbar \arrow[r,"\psibar"] \arrow[d] & \mathbb{R} \arrow[r] & 0 \\
    & 0                                   &  0                                            &                                        & \\
  \end{tikzcd}
\end{equation}
where $\ker\psi$ is carrollian and $\g_0$ metric.  We will now show
that $\gbar$ is galilean.

Let us split the exact sequences~\eqref{eq:carr-in-g} and
\eqref{eq:metric-in-gbar} by choosing $D \in \g$ with $\psi(D)=1$.
Let $\Dbar = \overline\pi(D)$ so that $\psibar(\Dbar) = 1$.
As vector spaces, $\g = \ker\psi \oplus \RR D$ and hence $\gbar = \g_0
\oplus \RR \Dbar$.

\begin{lemma}
  $\ad_D$ induces $\ad_{\gbar} \Dbar$ which in turn induces a
  skew-symmetric derivation $\Dbar_0$ of $(\g_0,\left<-,-\right>_0)$
  defined by
  \begin{equation*}
    \Dbar_0 \Xbar := [\Dbar,\Xbar] = \overline{[D,X]},
  \end{equation*}
  where the first bracket is in $\gbar$.
\end{lemma}

\begin{proof}
  The endomorphism $\Dbar_0$ is well-defined because $[D,X] \in [\g,\g]
  \subset \ker \psi$ and $Z$ is central.  We show that it is a
  derivation.  Let $X,Y \in \ker \psi$ and calculate
  \begin{align*}
    \Dbar_0 [\Xbar, \Ybar]_0 &= \Dbar_0 \overline{[X,Y]}\\
                                     &= \overline{[D,[X,Y]]}\\
                                     &= \overline{[[D,X],Y] +  [X,[D,Y]]} & \tag{by Jacobi}\\
                                     &= \overline{[[D,X],Y]} + \overline{[X,[D,Y]]}\\
                                     &= [\overline{[D,X]},\Ybar]_0 + [\Xbar, \overline{[D,Y]}]_0\\
                                     &= [\Dbar_0\Xbar, \Ybar]_0 +  [\Xbar, \Dbar_0 \Ybar]_0.
  \end{align*}
  To show that it is skew-symmetric we let $X,Y \in \ker \psi$ and calculate
  \begin{align*}
    \left<\Dbar_0\Xbar, \Ybar\right>_0 &= \left<\overline{[D,X]}, \Ybar\right>_0\\
                                          &= h([D,X],Y)\\
                                          &= - h(X,[D,Y]) & \tag{since $h$ is invariant}\\
                                          &= - \left<\Xbar, \overline{[D,Y]}\right>_0\\
                                          &= - \left<\Xbar, \Dbar_0 \Ybar\right>_0.
  \end{align*}
\end{proof}

The covector $\psibar \in \gbar^*$, which is induced from $\psi$, is
$\gbar$-invariant since $\psi$ is $\g$-invariant.  The inverse of
the $\g_0$-invariant scalar product on $\g_0 = \ker \psibar$ pushes
forward to a symmetric tensor $\gamma \in \odot^2\gbar$. Since the
derivation $\Dbar_0$ of $\g_0$ induced by $\ad_{\gbar}\Dbar$ is
skew-symmetric, it follows that $\gamma$ is actually
$\gbar$-invariant.  In other words, $(\gbar, \psibar, \gamma)$ is a
galilean Lie algebra.

\subsection{A leibnizian Lie algebra which is not bargmannian}
\label{sec:leibn-lie-algebra}

Comparing the commutative diagrams \eqref{eq:comm-diag-four} and
\eqref{eq:comm-diag-three}, we see that both $\ker\psi$ and
$\g_{\text{car}}$ are carrollian, $\g_0$ is metric in both diagrams,
$\gbar$ and $\g_{\text{gal}}$ are both galilean, but whereas $\ghat$
is bargmannian, $\g$ is leibnizian.  It would be tempting to
conjecture that leibnizian Lie algebras are actually bargmannian, but
this turns out to be false.

In diagram \eqref{eq:comm-diag-four} there are two skew-symmetric
derivations of the metric Lie algebra $\g_0$ at play: the derivation
$D_0$ associated to the central extension defining $\ker\psi$ and the
derivation $\Dbar_0$ by which we extend $\g_0$ in order to construct
$\gbar$.  If $\g$ were bargmannian, and hence the double extension of
$\g_0$ by a skewsymmetric derivation, both of these derivations would
coincide.  In the general leibnizian case, however, they need not
coincide and this gives us a hint how to prove that the class of
bargmannian Lie algebras is a proper subclass of the leibnizian Lie
algebras.

We will do this by constructing a leibnizian Lie algebra which is not 
bargmannian.  Let us take $\g_0$ to be abelian with a euclidean inner
product $\left<-,-\right>_0$.  Let $D_0 \in \so(\g_0)$ and define a
Lie algebra $\g_{\text{car}}$ with underlying vector space $\g_0
\oplus \RR Z$ and Lie brackets
\begin{equation}
  [X, Y] = \left<D_0 X, Y\right>_0 Z \qquad\text{and $Z$ central,}
\end{equation}
for all $X,Y \in \g_0$.  Pick $D \in \so(\g_0)$ in a different
$\SO(\g_0)$-orbit than $D_0$.  Let $\g$ be the Lie algebra with
underlying vector space $\g_0 \oplus \RR Z \oplus \RR W$ with Lie
brackets
\begin{equation}
  \begin{split}
    [X,Y] &=  \left<D_0 X, Y\right>_0 Z\\
    [W,X] &= DX\\
    [Z,-] &= 0,
  \end{split}
\end{equation}
for all $X,Y \in \g_0$.  We will show one example of this construction
for which $\g$ is not bargmannian.  Let $\g_0$ be four-dimensional,
with orthonormal basis $e_i$, $i=1,\dots,4$ and let $D_0,D \in \so(\g_0)$
be given relative to this basis by the matrices
\begin{equation}
  D_0 = \begin{pmatrix}
    \zero & 1 & \zero & \zero \\
    -1 & \zero & \zero & \zero \\
    \zero & \zero & \zero & \alpha\\
    \zero & \zero & -\alpha & \zero
  \end{pmatrix}
  \qquad\text{and}\qquad
  D = \begin{pmatrix}
    \zero & \beta & \zero & \zero \\
    -\beta & \zero & \zero & \zero \\
    \zero & \zero & \zero & \gamma\\
    \zero & \zero & -\gamma & \zero
  \end{pmatrix}.
\end{equation}
Let $\g$ be the six-dimensional Lie algebra spanned by
$e_1,\dots,e_4,e_+,e_-$ with nonzero brackets
\begin{equation}
  [e_i, e_j] = (D_0)_{ij} e_+ \qquad\text{and}\qquad [e_i, e_-] =
  D_{ij} e_j,
\end{equation}
or explicitly,
\begin{equation}
  \begin{aligned}\relax
    [e_1,e_2] &= e_+ \\
    [e_1,e_-] &= \beta e_2\\
    [e_2,e_-] &= -\beta e_1\\
  \end{aligned}
  \qquad\text{and}\qquad
  \begin{aligned}\relax
    [e_3,e_4] &= \alpha e_+\\
    [e_3,e_-] &= \gamma e_4\\
    [e_4,e_-] &= -\gamma e_3\\
  \end{aligned}
\end{equation}
Let $\theta^1,\dots,\theta^4,\theta^+, \theta^-$ denote the canonical
dual basis for $\g^*$.  Then the Lie algebra $\g$ is leibnizian with
$Z = e_+$, $\psi = \theta^-$ and $h = \sum_{i=1}^4 (\theta^i)^2$.  A
short calculation shows that the most general ad-invariant symmetric
bilinear form on $\g$ is given up to scale by
\begin{equation}
  \begin{pmatrix}
    \gamma & \zero & \zero & \zero & \zero & \zero\\
    \zero & \gamma & \zero & \zero & \zero & \zero\\    
    \zero & \zero & \alpha\beta & \zero & \zero & \zero\\
    \zero & \zero & \zero & \alpha\beta & \zero & \zero\\ 
    \zero & \zero & \zero & \zero & \lambda & -\beta\gamma\\
    \zero & \zero & \zero & \zero & -\beta\gamma & \zero\\
  \end{pmatrix}
\end{equation}
for any $\lambda \in \RR$.  This matrix is non-degenerate if and
only if $\alpha \beta \gamma \neq 0$.  So we can take either of
$\alpha$, $\beta$ or $\gamma$ to be zero (with, say, the other two
nonzero) and arrive at a leibnizian Lie algebra which is not
bargmannian.

\section{Conclusion}
\label{sec:conclusion}

In this paper we have characterised Lie algebras admitting
ad-invariant bargmannian, carrollian, and galilean structures.  We
have shown that they are all given by the same data: namely a metric
Lie algebra and a skew-symmetric derivation.  This sets up a canonical
correspondence between carrollian and galilean Lie algebras.

We have shown that these Lie algebras, together with the underlying
metric Lie algebra, fit into a commutative diagram
\eqref{eq:comm-diag-three} of Lie algebras which is isomorphic to a
similar diagram \eqref{eq:comm-diag} involving the Bargmann, Carroll
and Galilei Lie algebras (despite them not being bargmannian,
carrollian nor galilean) and also a similar diagram
\cite[App.~B]{Figueroa-OFarrill:2022tlf} involving some
infinite-dimensional Lie algebras playing a rôle in the hamiltonian
description of particle dynamics on the homogeneous galilean
spacetimes, whose transitive Lie algebras together with their Bargmann
extension and the Carroll Lie algebra itself also fit in yet another
commutative diagram \eqref{eq:comm-diag-too} isomorphic to the other
commutative diagrams.  This diagrammatic coincidence, if indeed it is
a coincidence, deserves to be further studied.

For the strict carrollian and galilean cases, the underlying metric
Lie algebra is positive-definite and this allows a classification up
to isometric isomorphism and the Lie algebras are orthogonal direct
sums of a compact Lie algebra with a choice of positive-definite inner
product and either a Nappi--Witten Lie algebra (in the bargmannian
case), a Heisenberg Lie algebra (in the carrollian case) or an
extension by skew-symmetric derivation of an abelian metric Lie
algebra (in the galilean case).

We also characterised Lie algebras admitting an ad-invariant
leibnizian structure and we saw that they too mediate a Lie algebraic
Carroll/Galilei correspondence which is however not canonical.  Indeed
there is some choice in defining the galilean dual of a carrollian Lie
algebra.  This is the phenomenon already observed at the level of the
spatially isotropic homogeneous galilean spacetimes.  Indeed, as we
could see already in diagram \eqref{eq:comm-diag-too}, the Carroll
spacetime (as the unique kinematical spacetime associated to the
Carroll algebra $\c$) is apparently dual to any one of the galilean
spacetimes simply because $\c$ is an ideal of any of the
$\b_{\alpha,\beta}$ which are extensions by (different) derivations of
$\c$.  Here too, given a carrollian Lie algebra $\g_{\text{car}}$ we
arrive at a metric Lie algebra $\g_0$ with a skew-symmetric derivation
$D_0$.  We may choose to extend $\g_0$ by this very derivation in
order to obtain the galilean dual Lie algebra $\g_{\text{gal}}$ in
diagram \eqref{eq:comm-diag-three} (with $\ghat$ bargmannian) or we
could choose any other skew-symmetric derivation and arrive in this
way at a leibnizian Lie algebra $\g$.  In a way, the bargmannian case
is canonical, since it is uniquely determined by the carrollian Lie
algebra, but the leibnizian case provides added flexibility in the
construction.

\section*{Acknowledgments}

I first reported some of these results at the 2nd Carroll Workshop in
Mons on 12 September 2022.  It is a pleasure to thank Nicolas
Boulanger, Andrea Campoleoni and Yannick Herfray for the invitation to
participate in the meeting.  It is also a pleasure to thank Kevin
Morand for a useful conversation during the conference and for some
very insightful correspondence afterwards; Quim Gomis for
interesting conversations about Carroll/Galilei duality; Stefan
Prohazka for comments and Dieter Van den Bleeken for a careful
reading of a previous version of this note; and Luminița Maxim for
the Romanian lesson.  Last, but by far not least, I would like to
dedicate this paper to the memory of Veronica Stanciu (1925-2022). May
she finally rest in peace.

\section*{Data availability}

Data sharing not applicable to this article as no datasets were
generated or analysed during the current study.

\appendix

\begin{appendices}

\section{Skew-symmetric derivations of a reductive metric Lie algebra}
\label{sec:skew-symm-deriv}

In this appendix we give a proof of
Proposition~\ref{prop:skew-sym-der}, which we rephrase slightly.

Let $\g_0 = \ses \oplus \a$ be an orthogonal direct sum of a compact
semisimple Lie algebra $\ses$ with a choice of a positive-definite
inner product (i.e., a negative multiple of the Killing form on each
simple summand) and an abelian Lie algebra $\a$ with a choice of
euclidean inner product.

\begin{proposition}
  Let $D$ be a skew-symmetric derivation of $\g_0$.  Then $D = D_1 +
  D_2$ where $D_1 = \ad_X$ for some $X \in \ses$ and $D_2 \in
  \so(\a)$.
\end{proposition}

\begin{proof}
  Let $X_1 \in \ses$ and $X_2 \in \a$.  Then since they commute and
  $D$ is a derivation
  \begin{equation*}
    0 = D[X_1,X_2] = [DX_1, X_2] + [X_1, DX_2].
  \end{equation*}
  Let $Y_1 \in \ses$ and calculate the inner product
  \begin{align*}
    0 &= \left<D[X_1,X_2],Y_1\right> \\
      &= \left<[DX_1,X_2] + [X_1, DX_2], Y_1\right>\\
      &= \left<DX_1, [X_2, Y_1]\right> - \left<DX_2, [X_1, Y_1]\right>.
  \end{align*}
  The first term in the RHS is zero because $[X_2, Y_1] = 0$ and hence
  the second term says that $DX_2$ is perpendicular to $[\ses,\ses] =
  \ses$.  In other words, $DX_2 \in \a$ and hence $D$ preserves $\a$.
  But then, since $\ses$ and $\a$ are orthogonal and $D$ is
  skew-symmetric,
  \begin{equation*}
    0 = \left<DX_1, X_2\right> + \left<X_1, DX_2\right>.
  \end{equation*}
  where the second term vanishes since $DX_2 \in \a = \ses^\perp$.
  Therefore $DX_1$ is perpendicular to $\a$ and hence $D$ also
  preserves $\ses$.  In other words, $D = D_1 + D_2$, where $D_1$ is a
  skew-symmetric derivation of $\ses$ and $D_2$ is a skew-symmetric
  derivation of $\a$.  Since $\ses$ is semisimple, all derivations are
  inner and since the inner product is built out of the Killing form,
  all inner derivations are also skew-symmetric, so that $D_1 = \ad_X$
  for some $X \in \ses$.  Finally, since $\a$ is abelian, any
  endomorphism is a derivation and hence $D_2 \in \so(\a)$ is any
  skew-symmetric endomorphism.
\end{proof}

Now Proposition~\ref{prop:skew-sym-der} follows because if $\ses =
\s_1 \oplus \cdots \oplus \s_k$, then $X \in \ses$ can be written as
$X = \sum_i X_i$.

\section{Once more, with indices}
\label{sec:once-more-with-indices}

In this appendix we will write down the bargmannian, carrollian and
galilean Lie algebras in terms of a basis.

Let $\g_0$ be a metric Lie algebra with basis $(X_a)$, Lie brackets
$[X_a, X_b] = f_{ab}{}^c X_c$ and scalar product $\left<X_a,
  X_b\right> = \eta_{ab}$.  Invariance of the inner product simply
says that $f_{abc}:= f_{ab}{}^d \eta_{dc}$ is totally skew-symmetric.

Let $D_0 : \g_0 \to \g_0$ be a skew-symmetric derivation.  Relative to
the above basis for $\g_0$, we write $D_0 X_a = X_b \omega^b{}_a$,
where the skew-symmetry condition says that
$\omega_{ab} = \eta_{ac}\omega^c{}_b = - \omega_{ba}$ and the fact
that $D_0$ is a derivation says that
\begin{equation}
  \omega^d{}_c f_{ab}{}^c = f_{cb}{}^d \omega^c{}_a + f_{ac}{}^d \omega^c{}_b.
\end{equation}

The bargmannian Lie algebra $\ghat$ associated to the data $(\g_0,
\left<-,-\right>, D_0)$ is the double-extension, which has basis
$(X_a, X_+, X_-)$ and Lie brackets
\begin{equation}
  [X_a, X_b] = f_{ab}{}^c X_c + \omega_{ab} X_+, \qquad [X_-, X_a] =
  X_b \omega^b{}_a \qquad\text{and}\qquad [X_+,-] = 0,
\end{equation}
and the invariant scalar product can always be brought to a form where
the only nonzero entries are:
\begin{equation}
  \left<X_a,X_b\right> = \eta_{ab} \qquad\text{and}\qquad
  \left<X_+,X_-\right>= 1.
\end{equation}

The associated carrollian Lie algebra $\g_{\text{car}}$ is the ideal
of $\ghat$ spanned by $(X_a, X_+)$ with induced brackets
\begin{equation}
  [X_a, X_b] = f_{ab}{}^c X_c + \omega_{ab} X_+ \qquad\text{and}\qquad
  [X_+, -] = 0.
\end{equation}
The invariant carrollian structure consists of $X_+$ and the invariant
symmetric bilinear form $\left<-,-\right>$ which is degenerate since
$\left<X_+,-\right> = 0$.

The associated galilean Lie algebra $\g_{\text{gal}}$ is the quotient
of $\ghat$ by the ideal spanned by $X_+$.  It is spanned by
$(X_a, X_-)$ and the brackets are
\begin{equation}
  [X_a, X_b] = f_{ab}{}^c X_c \qquad\text{and}\qquad [X_-,X_a] = X_b \omega^b{}_a.
\end{equation}
Let $(\theta^a, \theta^-)$ be a canonical dual basis for
$\g_{\text{gal}}^*$.  Then the invariant galilean structure is given
by $\theta^-$ and $\gamma = \eta^{ab}X_a X_b$, where $\eta^{ab}$ are
the entries of the inverse of the restriction of $\eta$ to $\g_0$.

In the special case where $\eta_{ab}$ is positive-definite we
discussed in Section~\ref{sec:classification}, we showed that these
Lie algebras are the orthogonal direct sums of certain `` primitive''
Lie algebras with a compact Lie algebra with a positive-definite
invariant inner product (i.e., the direct sum of a compact semisimple
Lie algebra with a choice of positive-definite invariant inner product
and an abelian Lie algebra with a euclidean inner product).  These
primitive Lie algebras can be read off from the above expressions by
setting $f_{ab}{}^c = 0$, $\eta_{ab}$ positive-definite and
$\omega_{ab}$ symplectic.

Explicitly, the bargmannian Lie algebra is a Nappi--Witten Lie algebra
with brackets
\begin{equation}
  [X_a, X_b] = \omega_{ab} X_+, \qquad [X_-, X_a] = X_b \omega^b{}_a
  \qquad\text{and}\qquad [X_+,-] = 0,
\end{equation}
with inner product $\left<X_a, X_b\right> = \eta_{ab}$ and
$\left<X_+,X_-\right>=1$.  Of course, we can always choose an
orthonormal basis where $\eta_{ab} = \delta_{ab}$ and then use
orthogonal transformations to bring $\omega_{ab}$ to a normal form
\begin{equation}
  \omega = \lambda_1 \theta^1 \wedge \theta^2 + \lambda_2 \theta^3
  \wedge \theta^4 + \cdots + \lambda_\ell \theta^{2\ell + 1} \wedge
  \theta^{2\ell}
\end{equation}
where $\lambda_1 \geq \lambda_2 \geq \cdots \geq \lambda_\ell \geq 0$,
where the Nappi--Witten Lie algebra has dimension $2\ell + 2$.

The carrollian Lie algebra is a Heisenberg Lie algebra with brackets
\begin{equation}
  [X_a, X_b] = \omega_{ab} X_+ \qquad\text{and}\qquad [X_+,-] = 0,
\end{equation}
and the only nonzero brackets of the galilean Lie algebra are
\begin{equation}
  [X_-, X_a] = X_b \omega^b{}_a.
\end{equation}

\end{appendices}

\providecommand{\href}[2]{#2}\begingroup\raggedright\endgroup


\end{document}